\documentclass [a4paper, twoside, 11pt]{article}
\usepackage[utf8]{inputenc}
\usepackage[T1]{fontenc}
\usepackage[english]{babel}
\usepackage{enumitem}

\usepackage[draft,english]{fixme}
\usepackage{cite}
\usepackage{yfonts}
\usepackage{mathtools}
\usepackage{amsmath, amssymb}
\usepackage{dsfont}
\usepackage[amsmath,thmmarks]{ntheorem} 
\theorembodyfont{\normalfont\slshape}
\theoremseparator{}
\theoremstyle{break}
\newtheorem{thm}{Theorem}[section]

\newtheorem{rem}[thm]{Remark}

\newtheorem{lemma}[thm]{Lemma}

\newtheorem{prop}[thm]{Proposition}

\theorembodyfont{\normalfont\itshape}

\theoremstyle{nonumberplain}
\theoremheaderfont{%
\normalfont\scshape}
\theorembodyfont{\normalfont}
\theoremsymbol{\ensuremath{\blacksquare}}
\theoremseparator{:}
\newtheorem{proof}{Proof}

\usepackage{keyval}
\makeatletter
\newcommand*\DeclareMathSymbolShorthand[2]{
   \begingroup
   \setkeys{DMSS}{name=#2,#1}%
   \if\DMSS@overwrite 
   \else
      \expandafter\@ifdefinable\csname \DMSS@prefix\DMSS@name\endcsname{%
        \def\DMSS@overwrite{00}
      }%
   \fi%
   \if\DMSS@overwrite 
      \expandafter\@firstofone%
   \else\expandafter\@gobble\fi%
   {\protected\expandafter%
        \xdef\csname \DMSS@prefix\DMSS@name \endcsname{%
        \unexpanded\expandafter{\DMSS@format{#2}}}}%
   \endgroup}
\define@key{DMSS}{format}{\let\DMSS@format#1}
\define@key{DMSS}{format*}{\def\DMSS@format{\expandafter#1\@firstofone}}
\define@key{DMSS}{name}{\def\DMSS@name{#1}}
\define@key{DMSS}{prefix}{\def\DMSS@prefix{#1}}
\define@key{DMSS}{overwrite}[true]{%
   \edef\DMSS@overwrite{\csname if#1\endcsname 00\else 01\fi}}
\setkeys{DMSS}{overwrite=false}
\setkeys{DMSS}{format*=}
\newcommand\MakeDeclareMathSetCommand[3]{%
   \expandafter\MakeDeclareShorthandCommandAux\csname math#2format\endcsname
   {#1}{#2}{#3}}
\def\MakeDeclareShorthandCommandAux#1#2#3#4{%
   \newcommand*#1{#4}%
   \newcommand*#2[2][]{%
      \DeclareMathSymbolShorthand{format=#1,prefix=#3,##1}{##2}%
   }}
\MakeDeclareMathSetCommand{\DeclareMathSet}{numbers}{\mathbb}
\makeatother

\makeatletter
\def\blfootnote{\xdef\@thefnmark{}\@footnotetext}
\makeatother

\MakeDeclareMathSetCommand{\DeclareMathNumbers}{numb}{\mathbb}
\MakeDeclareMathSetCommand{\DeclareMathGroup}{group}{\mathrm}
\DeclareMathNumbers{N} 
\DeclareMathNumbers{Z} 
\DeclareMathNumbers{Q} 
\DeclareMathNumbers{R} 
\DeclareMathNumbers{C} 
\DeclareMathGroup{GL} 
\DeclareMathGroup{SL} 
\DeclareMathGroup{Mat} 

\newcommand{\U}[0]{\mathbf{U}}

\newcommand{\GL}[0]{\mathbf{GL}}
\newcommand{\SetH}[0]{\mathds{H}}

\newcommand{\tr}{\text{tr}}
\newcommand{\T}{\mathcal{T}}
\newcommand{\M}{\mathcal{M}}

\renewcommand{\phi}{\varphi}
\renewcommand{\epsilon}{\varepsilon}

\DeclareMathOperator{\Id}{Id}
\newcommand{\ra}{\rightarrow}
\title{An explicit Ricci potential for the Universal Moduli Space of Vector Bundles}
\author{Jørgen Ellegaard Andersen and Niccolo Skovgård Poulsen}

\begin{document}
\blfootnote{Research supported in part by the center of excellence grant ``Center for Quantum Geometry of Moduli Spaces (QGM)'' DNRF95 from the Danish National Research Foundation.}

\maketitle
\begin{abstract}
In this paper we modify the coordinate construction in \cite{AP} on the universal moduli space of pair consisting of a Riemann Surfaces and a stable holomorphic bundles on the Riemann Surface, so as to produce a new set of coordinates, which are in fact K\" ahler coordinates on this universal moduli space. Further, we give a functional determinant formula for the Ricci potential of the universal moduli space.
\end{abstract}

\begin{center}
\thanks{Dedicated to Nigel Hitchin at the conference {\em Hitchin70},\\ celebrating his 70'th Birthday.}
\end{center}

\section{Introduction}
In this paper we modify the coordinates construction from  \cite{AP} for $\M'$, the smooth part of the universal moduli space of pairs of a Riemann surface and a rank $n$ and degree $k$ holomorphic vector bundle on the Riemann surface. The coordinates in \cite{AP} are not K{\"a}hler coordinates. To construct coordinates in a neighborhood of any $(X,E)\in \M'$, we follow the method presented in \cite{AP}, which in tern is a generalisation of the pioneering coordinate constructions of Takhtajan and Zograf \cite{ZTVB, ZTPVB,ZTpuncRie, ZTRie},  namely we seek a smooth family of principal bundle maps
$$\chi^{\mu\oplus\nu} : \SetH\times\mathbf{GL}(n,\numbC)\to\SetH\times\mathbf{GL}(n,\numbC)$$ 
parametrised by $\mu\oplus\nu$ in a small neigborhood of $0$ in $H^1(X,TX)\oplus H^1(X,\text{End}E)$. For each $\mu\oplus\nu$ we the get a corresponding point in $\M'$ given by the formula
  \begin{align}
\label{eqeq}
   (\rho_\SetH^{\mu\oplus\nu},\rho_E^{\mu\oplus\nu})(\gamma)=(\chi^{\mu\oplus\nu}\circ \gamma) (\rho_\SetH,\rho_E)(\gamma)(\chi^{\mu\oplus\nu})^{-1}.
  \end{align}
 As a smooth manifolds, we recall that
$$ 
{\mathcal M}' = \mathcal{T}\times M',
$$
by the Narasimhan-Seshadri theorem, where $\mathcal{T}$ is Teichm\"{u}ller space and $M'$ is the moduli space of flat irreducible $U(n)$-connections with holonomy $e^{2\pi i k/n}\Id$ around a marked point on the surface and say $(\rho_{\mathbb H},\rho_E)$ correspond to $(X,E)$ under this identification.
Hence, we see that there is a natural symplectic structure on ${\mathcal M}'$, namely
$$\omega_{{\mathcal M}'} = p_{\mathcal T}^*\omega_{\mathcal{T}} + p_{M'}^*\omega_{M'},$$
where $ \omega_{\mathcal{T}}$ is the Weil-Petersen symplectic form on $\mathcal{T}$, $\omega_{M'}$ is the Seshadri-Atiyah-Bott-Goldman symplectic form on $M'$, $p_{\mathcal T}$ is the projection onto ${\mathcal T}$ and $p_{M'}$ is the projection onto $M'$.

In this paper we prove the following theorem, which gives us local coordinates around any pair $(\rho_{\mathbb H},\rho_E)\in {\mathcal T}\times M'$.
\begin{thm}
\label{thm:uni coord}
For all sufficiently small $\mu\oplus\nu\in H^{1}(X, TX)\oplus H^{1}(X,\text{End}E)$ there exist a unique bundle map $\chi^{\mu\oplus\nu}$ such that
\begin{enumerate}
\label{thm:uni coord}
\item The bundle map $\chi^{\mu\oplus\nu}$ solves 
\begin{align}\label{eqM}
  \bar\partial_\SetH \chi^{\mu\oplus\nu}=(\mu-\frac{1}{2}\tilde g^{-1}_X\tr\nu\otimes\nu)\cdot\partial_\SetH\chi^{\mu\oplus\nu}+\partial_{\mathbf{GL}(n,\numbC)}\chi^{\mu\oplus\nu}\cdot\nu
\end{align} 
where $\nu$ is considered a left-invariant vector field on $\GL(n,\numbC)$ at each point in $\SetH$ and  $\tilde{g}_X$ is the hyperbolic metric of $X$.
\item The base map extends to the boundary of $\SetH$ and fixes $0,1$ and $\infty$.
\item The pair of representations $(\rho_\SetH^{\mu\oplus\nu},\rho^{\mu\oplus\nu}_E)$ defined by equation (\ref{eqeq}) represents a point in $\mathcal{T}\times M'$.
\item $p_{\GL(n,\numbC)}(\chi^{\mu\oplus\nu}(z_0,e))$ has determinant $1$ and is positive definite. 
\end{enumerate}
\end{thm}

We remark that the conditions in this theorem are identical to the ones in Theorem 1.1 in \cite{AP}, except for the term $-\frac{1}{2}\tilde g^{-1}_X\tr\nu\otimes\nu$ in equation (\ref{eqM}). We establish that even though we add this term, we still get complex coordinates in a neigborhood of $(\rho_{\mathbb H},\rho_E)$, and a further calculation show that these are indeed K\"ahler coordinates as opposed to the coordinates introduced in \cite{AP}\footnote{We would like to thank Peter Zograf for asking us if the coordinates in \cite{AP} are K\"{a}hler or not.}. We summarice this in the following theorem.

\begin{thm}
For all pairs $(\rho_{\mathbb H},\rho_E)\in {\mathcal T}\times M'$, the coordinates
$$(\mu,\nu) \mapsto (\rho_\SetH^{\mu\oplus\nu},\rho^{\mu\oplus\nu}_E),$$
for $(\mu,\nu)$ runing in a certain open neighbourhood of zero in $H^{1}(X, TX)\oplus H^{1}(X,{\rm End}E)$ are local K\"{a}hler coordinates on  $({\mathcal M}', \omega_{{\mathcal M}'})$.
\end{thm}
\begin{rem}
  If we consider our coordinates based at a point where $\det E=\mathcal{O}$ then restricting the construction of coordiantes to $H^{1}(X, TX)\oplus H^{1}(X,{\rm Ad}E)$ will give coordinates on the moduli space of pairs of Riemann surfaces and holomorphic vector bundles with trivial determinant, ${\mathcal M}'_{\mathcal{O}}$.
\end{rem}
Let us denote the K\"{a}hler metric on $({\mathcal M}', \omega_{{\mathcal M}'})$ by $g_{{\mathcal M}'}$. The Ricci $(1,1)$-form of this metric is denote by $\text{Ric}^{1,1}$. Further we let $\Delta_0$ denote the Laplacian acting on $1$-forms on $X$ and let $\Delta_{\text{Ad}E}$ be the Laplacian acting on $1$-forms twisted by the bundle $\text{Ad}E$ with the flat connection induced from $\rho_E$.
We now consider the function $F\in C^\infty({\mathcal M}')$ given by
\begin{equation}\label{RP}
F(\rho_{\mathbb H},\rho_E) = \frac{1}{2}\log\det\Delta_{\text{Ad}E}\det\Delta_0.
\end{equation}

The second main result of this paper is that $F$ is a Ricci potential 
of $({\mathcal M}'_{\mathcal{O}}, g_{{\mathcal M}_{\mathcal{O}}'})$.
\begin{thm}\label{thmRP}
  The function $F$ is a Ricci potential for the metric $g_{{\mathcal M}'}$ on $\mathcal{M}_{\mathcal{O}}'$, e.g. it  fulfils the following equation 
  \begin{align*}
    2i\partial\bar\partial F=\text{Ric}^{1,1}-\frac{n}{2\pi}\omega_{M'}-\frac{n^2}{12\pi}\omega_{\mathcal{T}}
  \end{align*}
In particular, the cohomology of the Ricci form is $\frac{n}{2\pi}\omega_{M'}+\frac{n^2}{12\pi}\omega_{\mathcal{T}}$.
\end{thm}

In our paper \cite{AP3} we will use of this formula to compute the curvature of the Hitchin connection.

\section{K{\"a}hler Coordinates for the Moduli Space of Pairs}
In order to construct coordinates we will map a small neighbourhood of zero of the tangent space onto the moduli space $\M'$ of pairs of a Riemann structure on $\Sigma$ and a rank $n$ degree $k$ stable vector bundle over this Riemann surface. We will in fact consider marked Riemann surfaces and thus this moduli space is diffeomorphic to the product of the Teichm\"uller space of $\Sigma$, $\mathcal{T}$, with the moduli space of irreducible flat connections on $\Sigma-p$ with holonomy around $p$ given by $e^{\frac{-2\pi i k}{n}}I$, $M'$. By the Narasimhan-Seshadri theorem, the space $M'$ can be given a complex structure $J_{[X]}$ by identifying it with the space of stable holomorphic bundles of rank $n$ and degree $k$, $M_{n,k}(X)$ for any $[X]\in \T$. This gives an almost complex structure which is in fact integrable \cite{JEANLGML}. As  we established in \cite{AP}, we have that these two complex structure are the same.
\begin{prop}
\label{prop:Csstrucmodspace}
The almost complex structure $J$ is in fact the complex analytic structure this space gets from the Narasimhan-Seshadri diffeomorphism
$$\Psi: (\T\times M', J) \ra \M'$$
e.g. this map is complex analytic.
\end{prop}
For the details of the proof see \cite{AP}. The idea of the proof is to construct a holomorphic family of Riemann surfaces, all with a bundle given by the same $\U(n)$-representation $\rho_E$. To do so, consider Teichm\"{u}ller space cartesian product with ${\mathbb H}\times \numbC^n$ and take the fiberwise sheaf theoretic quotient with respect to the action of $\rho_{\mathbb H}\times\rho_E$ for $\rho_{\mathbb H}$  running through a component of  $\hom(\pi_1,PSL(2,{\mathbb R}))$ which corresponds to Teichm\"{u}ller space. After division by the conjugation action of $PSL(2,{\mathbb R})$, this family let us split the tangent space at each point in the moduli space of pairs and identify the holomorphic tangent space of the moduli space of pairs with $H^1(X,TX)\oplus H^1(X,\text{End}E)$ at $(X,E)\in\mathcal{M}$. In what follows we will always identify these spaces with harmonic $(0,1)$-forms.

We recall that we have chosen the representation $\rho_{\mathbb H}$ giving the Riemann surface $X$ and  we can think of the representation $\rho_E$ as an irreducible $U(n)$-representation of $\pi_1^{orb}(X_n)$, where the elliptic element is mapped to $e^{-\frac{2 \pi k}{n}}I$ and where $X_n$ is the orbifold cover of $X$ completely ramified over a single point with ramification index $n$. We consider $\SetH$ as the cover of $X_n$ and pick  $z_0$ as a point in $\SetH$ covering our ramified point in $X_n$. We will now proceed to construct the maps $\chi^{\mu\oplus\nu} : \SetH\times \mathbf{GL}(n,\numbC)\to\SetH\times \mathbf{GL}(n,\numbC)$ for each $\mu\oplus\nu$ in a small neigborhood of $0$ in $ H^1(X,TX)\oplus H^1(X,\text{End}E)$. We then define a new point in $\M$ by (\ref{eqeq}).

Compared to \cite{AP} we have introduced the additional term $\frac{-1}{2}\tilde g^{-1}_X\tr\nu\otimes\nu\partial_\SetH,$ in the equation for $\chi^{\mu\oplus\nu}$, where $\tilde g_X$ will be used both as the metric tensor and in integrals as the metric density.

We construct the maps $\chi^{\mu\oplus\nu}$ as follows.
\begin{proof}[Theorem \ref{thm:uni coord}]
    These equations split in two equations, namely one for $\chi_1^{\mu\oplus\nu}:\SetH\to\SetH$ and one for $\chi^{\mu\oplus\nu}_2:\SetH\times \mathbf{GL}(n,\numbC)\to\mathbf{GL}(n,\numbC)$ as follows
\begin{align}
  \bar\partial_\SetH \chi_1^{\mu\oplus\nu}&=(\mu-\frac{1}{2}\tilde g^{-1}_X\tr\nu\otimes\nu)\cdot\partial_\SetH\chi_1^{\mu\oplus\nu} \label{eq:def1}
\\  \bar\partial_\SetH \chi_2^{\mu\oplus\nu}&=(\mu-\frac{1}{2}\tilde g^{-1}_X\tr\nu\otimes\nu)\cdot\partial_\SetH\chi_2^{\mu\oplus\nu}+\partial_{\mathbf{GL}(n,\numbC)}\chi_2^{\mu\oplus\nu}\cdot\nu. \label{eq:def2}
\end{align}
Since $\mu-\frac{1}{2}\tilde g^{-1}_X\tr\nu\otimes\nu$ is an analytic family of Beltrami differential, Bers has shown we can solve the first equation and that the suitably normalised solution depends analytic on the family of Beltrami differentials. The second equation is solved by first finding the antiholomorphic solution to \[  \bar\partial_\SetH \chi_-^{\mu\oplus\nu}=\chi_-^{\mu\oplus\nu}\cdot\nu.\] This can be done, since $\SetH$ is simply connected and $\nu$ is a harmonic $(0,1)$-form and so the equation (\ref{eq:def2}) can be thought of as the equation for a trivialization for a flat connection, where $\chi_-^{\mu\oplus\nu}$ is the gauge transformation relating this connection to the trivial one on  $\SetH$. Then we define the representation $$r^{\mu\oplus\nu}(\gamma)=\chi_-^{\mu\oplus\nu}(\gamma z)\rho_E(\gamma)(\chi_-^{\mu\oplus\nu}(z))^{-1} .$$ Since $E$ is stable, we conclude for $\mu\oplus\nu$ small that $r^{\mu\oplus\nu}$ define a stable bundle on $X_{\mu\oplus\nu}(=\SetH/\rho_\SetH^{\mu\oplus\nu})$. This means we can find a holomorphic gauge transformation, $\chi_+^{\mu\oplus\nu}$ on $X_{\mu\oplus\nu}$ so as to make the representation $r^{\mu\oplus\nu}$ admissible. Now using the defining differential equations we find that $\chi_+^{\mu\oplus\nu}\circ\chi_1^{\mu\oplus\nu}\in\ker(\bar\partial_\SetH-(\mu-\frac{1}{2}\tilde g^{-1}_X\tr \nu^2)\partial)$. Also since it is independent of $\mathbf{GL}(n,\numbC)$ we see that $(\chi_+^{\mu\oplus\nu}\circ\chi_1^{\mu\oplus\nu})\chi_-^{\mu\oplus\nu}$ solve \eqref{eq:def2}. And we have that
\begin{align*}
  \chi_+^{\mu\oplus\nu}\circ&\chi_1^{\mu\oplus\nu}\chi_-^{\mu\oplus\nu}(\rho_\SetH(\gamma) z,e)\rho_E(\gamma)(\chi_+^{\mu\oplus\nu}\circ\chi_1^{\mu\oplus\nu}\chi_-^{\mu\oplus\nu}(z,e))^{-1}\\&=\chi_+^{\mu\oplus\nu}(\chi_1^{\mu\oplus\nu}(\rho_\SetH (\gamma)( \chi_1^{\mu\oplus\nu})^{-1}z),e)r^{\mu\oplus\nu}(\gamma)\chi_+^{\mu\oplus\nu}(z,e))^{-1}=\rho^{\mu\oplus\nu}(\gamma),
\end{align*}
since the conjugation by the gauge transformation does not depend on which base point we chose, and so we choose $ (\chi_1^{\mu\oplus\nu})^{-1}(z)$ in the second to last equality instead of $z$.
 
So we get an admissible representation, finally we can normalise the choises by requireing that $\chi_1$ fix $0,1,\infty$ and $\chi_2(z_0)$ is a positive definite matrix of determinant $1$ at $z_0$.
\end{proof}
By the implicit function theorem it follow that
\begin{thm}
  The assignment $\mu\oplus\nu \mapsto (\rho^{\mu\oplus\nu}_\SetH,\rho_E^{\mu\oplus\nu})$ gives complex analytic coordinates for $\mathcal{M}'$ in a small neighboorhood of $([X],[E])\in {\mathcal M}'$.
\end{thm}
Alternatively this can be seen from the calculation of the Kodaira-Spencer maps in the next section, which shows the coordinates are in fact holomorphic coordinates.

\subsection{The Kodaira-Spencer Map}

Let us now consider the Kodaira-Spencer map at any point of our coordinates.

\begin{lemma}
  The Kodaira-Spencer map in our coordinates is given by
  \begin{multline*}
    KS_{\mu\oplus\nu}(\mu_1\oplus\nu_1)=P_{TX}\left((\chi_1^{\mu\oplus\nu})^{-1}_*\frac{\mu_1-\tilde g^{-1}_{X_{\mu\oplus\nu}}\tr\nu_1\nu}{1-\vert\mu\vert^2}\right)\oplus \\P_{\textnormal{End}E}\left((\chi_1^{\mu\oplus\nu})^{-1}_*\text{Ad}\chi_2^{\mu\oplus\nu}\left(\nu_1+(\mu_1-\tilde g^{-1}_{X_{\mu\oplus\nu}}\tr\nu_1\nu)\chi_2^{\mu\oplus\nu}\partial_\SetH\chi_2^{\mu\oplus\nu}\right)\right)
  \end{multline*}
when we identify $H^1(X,TX)\oplus H^1(X,\textnormal{End}E)$, with the harmonic $(0,1)$-forms. Here $P_{TX}$ is the projection on the harmonic Beltrami differentials on $X$ and $P_{\textnormal{End}E}$ is the projection on harmonic $(0,1)$-forms with values in $\textnormal{End}E$.
\end{lemma}
\begin{proof}
  The Kodaira-Spencer class as a harmonic Beltrami differential 
 for the first factor is given by the harmonic projection of
  \begin{align*}
    \bar\partial_\SetH &\left.\frac{d}{d\epsilon}\right\vert_{\epsilon=0}\chi_1^{\epsilon(\mu_1\oplus\nu_1)+\mu\oplus\nu}\circ(\chi_1^{\mu\oplus\nu})^{-1}\\ &=\left.\frac{d}{d\epsilon}\right\vert_{\epsilon=0}((\partial\chi_1^{\epsilon(\mu_1\oplus\nu_1)+\mu\oplus\nu})\circ(\chi_1^{\mu\oplus\nu})^{-1}\bar\partial (\chi_1^{\mu\oplus\nu})^{-1}\\ &+(\bar\partial\chi_1^{\epsilon(\mu_1\oplus\nu_1)+\mu\oplus\nu})\circ(\chi_1^{\mu\oplus\nu})^{-1}\overline{\partial (\chi_1^{\mu\oplus\nu})^{-1}}.
  \end{align*}
Now we rewrite this using the differential equation $$\bar\partial \chi_1^{\mu\oplus\nu}=(\mu-\frac{1}{2}g^{-1}_X\tr(\nu)^2)\partial\chi_1^{\mu\oplus\nu}.$$ This equation also imply that $$\bar\partial(\chi_1^{\mu\oplus\nu})^{-1}=-(\mu-\frac{1}{2}\tilde g^{-1}_X\tr \nu^2))\circ(\chi_1^{\mu\oplus\nu})^{-1}\partial (\chi_1^{\mu\oplus\nu})^{-1},$$
which let us conclude that
  \begin{align*}
    \bar\partial_\SetH &\left.\frac{d}{d\epsilon}\right\vert_{\epsilon=0}\chi_1^{\epsilon(\mu_1\oplus\nu_1)+\mu\oplus\nu}\circ(\chi_1^{\mu\oplus\nu})^{-1}\\ &=\left.\frac{d}{d\epsilon}\right\vert_{\epsilon=0}(-(\partial\chi_1^{\epsilon(\mu_1\oplus\nu_1)+\mu\oplus\nu}(\mu-\frac{1}{2}\tilde g^{-1}_X\tr \nu^2))\circ(\chi_1^{\mu\oplus\nu})^{-1}\partial (\chi_1^{\mu\oplus\nu})^{-1}\\ &+(\mu+\epsilon\mu_1-\frac{1}{2}g^{-1}_X\tr(\nu+\epsilon\nu_1)^2)\partial\chi_1^{\epsilon(\mu_1\oplus\nu_1)+\mu\oplus\nu})\circ(\chi_1^{\mu\oplus\nu})^{-1}\overline{\partial (\chi_1^{\mu\oplus\nu})^{-1}}.
  \end{align*}
Now all terms contain $\epsilon$ and so it is clear that the derivativ is
  \begin{align*}
    \bar\partial_\SetH &\left.\frac{d}{d\epsilon}\right\vert_{\epsilon=0}\chi_1^{\epsilon(\mu_1\oplus\nu_1)+\mu\oplus\nu}\circ(\chi_1^{\mu\oplus\nu})^{-1}\\ &=\left(\frac{\mu_1-g^{-1}_X\tr(\nu\nu_1)}{1-\vert\mu\vert^2}\frac{\partial\chi_1^{\mu\oplus\nu}}{\overline{\partial \chi_1^{\mu\oplus\nu}}}\right)\circ(\chi_1^{\mu\oplus\nu})^{-1}.
  \end{align*}

For the second part, we have that the Kodiara-Spencer class is the harmonic representative of (for details see Lemma 3.1. of \cite{AP})
\begin{align*}
  \bar\partial (\frac{d}{d \epsilon}\vert_{\epsilon=0}&\chi_2^{\epsilon(\mu_1\oplus\nu_1)+\mu\oplus\nu}(\chi_2^{\mu\oplus\nu})^{-1})\circ(\chi_1^{\mu\oplus\nu})^{-1}\\ &=\frac{d}{d \epsilon}\vert_{\epsilon=0}\text{Ad}((\chi_2^{\mu\oplus\nu})\circ(\chi_1^{\mu\oplus\nu})^{-1}((\chi_2^{\epsilon(\mu_1\oplus\nu_1)+\mu\oplus\nu})^{-1}\circ(\chi_1^{\mu\oplus\nu})^{-1}\\ &\qquad \bar\partial( \chi_2^{\epsilon(\mu_1\oplus\nu_1)+\mu\oplus\nu}\circ(\chi_1^{\mu\oplus\nu})^{-1}))).
\end{align*}
We use the differential equations for $\chi_2^{\mu\oplus \nu}$ and $\chi_1^{\mu\oplus\nu}$ to see that
\begin{align*}
  \bar\partial (\frac{d}{d \epsilon}\vert_{\epsilon=0}&\chi_2^{\epsilon(\mu_1\oplus\nu_1)+\mu\oplus\nu}(\chi_2^{\mu\oplus\nu})^{-1})\circ(\chi_1^{\mu\oplus\nu})^{-1}\\ &=\frac{d}{d \epsilon}\vert_{\epsilon=0}(\text{Ad}((\chi_2^{\mu\oplus\nu})((\chi_2^{\epsilon(\mu_1\oplus\nu_1)+\mu\oplus\nu})^{-1}\\ &\qquad (\epsilon\mu_1-\epsilon\tilde g^{-1}_\sigma\tr\nu_1\nu-\frac{\epsilon^2}{2}\tr\nu_1^2)\partial \chi_2^{\epsilon(\mu_1\oplus\nu_1)+\mu\oplus\nu}\\ &\qquad +\chi_2^{\epsilon(\mu_1\oplus\nu_1)+\mu\oplus\nu}(\nu+\epsilon\nu_1))\circ(\chi_1^{\mu\oplus\nu})^{-1}(\overline{\partial(\chi_1^{\mu\oplus\nu})^{-1}}).
\end{align*}
Now again all terms contain a factor of $\epsilon$ and so the derivativ is
\begin{align*}
  \bar\partial (\frac{d}{d \epsilon}\vert_{\epsilon=0}&\chi_2^{\epsilon(\mu_1\oplus\nu_1)+\mu\oplus\nu}(\chi_2^{\mu\oplus\nu})^{-1})\circ(\chi_1^{\mu\oplus\nu})^{-1}\\ &=(\chi_1^{\mu\oplus\nu})^{-1}_*\text{Ad}\chi_2^{\mu\oplus\nu}(\nu_1+(\mu_1-\tilde g^{-1}_{X_{\mu\oplus\nu}}\tr\nu_1\nu)\chi_2^{\mu\oplus\nu}\partial_\SetH\chi_2^{\mu\oplus\nu}).
\end{align*}
\end{proof}
\subsection{Properties of $\chi^{\mu\oplus\nu}$ and derivatives of families of operators.}
\label{sec:thetaprop}
To understand the metric and to find a Ricci potential, we differentiate our coordinate functions. We have the following formulae
\begin{lemma}\label{lemchiprop}
  The following equations hold
  \begin{align*}
    \frac{d}{d\epsilon}\vert_{\epsilon=0}\left((\overline{\chi_2^{\epsilon(\mu\oplus\nu)}})^T\chi_2^{\epsilon(\mu\oplus\nu)}\right)&=0, \\
    \frac{d}{d\epsilon}\vert_{\epsilon=0}\partial\chi_2^{\epsilon(\mu\oplus\nu)}&=-\bar\nu^T.
  \end{align*}
\end{lemma}
\begin{proof}

  To show the first equality, we first observe that by definition we have that
  \begin{align*}
    \rho^{\mu\oplus\nu}(\gamma)=(\chi_2^{\mu\oplus\nu}\circ\rho_\SetH(\gamma))^{-1}\rho_E(\gamma) \chi_2^{\mu\oplus\nu}\Rightarrow \chi_2^{\mu\oplus\nu}\circ\rho_\SetH(\gamma)=\rho_E(\gamma) \chi_2^{\mu\oplus\nu} \rho^{\mu\oplus\nu}(\gamma)^{-1}
  \end{align*}
thus we get that
$$((\overline{\chi_2^{\epsilon(\mu\oplus\nu)}})^T\chi_2^{\epsilon(\mu\oplus\nu)})\circ \rho_\SetH(\gamma)=\text{Ad}\rho^{\mu\oplus\nu}(\gamma)(\overline{\chi_2^{\epsilon(\mu\oplus\nu)}})^T\chi_2^{\epsilon(\mu\oplus\nu)}.$$ 
We observe that $(\overline{\chi_2^{\epsilon(\mu\oplus\nu)}})^T\chi_2^{\epsilon(\mu\oplus\nu)}$ is a section of End$E_{\rho_E^{\epsilon(\mu\oplus\nu)}}$ over $X_0$. Next we calculate
\begin{align*}
  \Delta_0\frac{d}{d\epsilon}\vert_{\epsilon=0}(\overline{\chi_2^{\epsilon(\mu\oplus\nu)}})^T\chi_2^{\epsilon(\mu\oplus\nu)}=&\Delta_0(\frac{d}{d\epsilon}\vert_{\epsilon=0}\overline{\chi_-^{\epsilon(\mu\oplus\nu)}}^T+\frac{d}{d\epsilon}\vert_{\epsilon=0}\overline{\chi_+^{\epsilon(\mu\oplus\nu)}\circ\chi_1^{\epsilon(\mu\oplus\nu)}}^T\\ &+\frac{d}{d\epsilon}\vert_{\epsilon=0}\chi_+^{\epsilon(\mu\oplus\nu)}\circ\chi_1^{\epsilon(\mu\oplus\nu)}+\frac{d}{d\epsilon}\vert_{\epsilon=0}\chi_-^{\epsilon(\mu\oplus\nu)}).
\end{align*}
Since the first and last term are harmonic, $\Delta_0$ annihilate them. For the two middle terms we use that $$(\bar\partial-(\epsilon\mu-\frac{1}{2}\tilde g_X\tr(\epsilon\nu)^2)\partial)\chi_+^{\epsilon(\mu\oplus\nu)}\circ\chi_1^{\epsilon(\mu\oplus\nu)}=0$$ and so we have that
\begin{align*}\Delta_0\frac{d}{d\epsilon}\vert_{\epsilon=0}(\overline{\chi_2^{\epsilon(\mu\oplus\nu)}})^T\chi_2^{\epsilon(\mu\oplus\nu)}= -\bar\mu&\bar\partial\bar\partial\overline{\chi_+^{0(\mu\oplus\nu)}\circ\chi_1^{0(\mu\oplus\nu)}}^T\\&-\mu\partial\partial{\chi_+^{0(\mu\oplus\nu)}\circ\chi_1^{0(\mu\oplus\nu)}}=0\end{align*}
since $\chi_+^{0(\mu\oplus\nu)}=I$.

Hence we get that $\frac{d}{d\epsilon}\vert_{\epsilon=0}(\overline{\chi_2^{\epsilon(\mu\oplus\nu)}})^T\chi_2^{\epsilon(\mu\oplus\nu)}$ must be a multiple of the identity though of as a section of the bundle End$E_{\rho^{\mu\oplus\nu}_E}$. From the determinant criterium we find that
\begin{equation*}
  0=\frac{d}{d\epsilon}\vert_{\epsilon=0}\det\big((\overline{\chi_2^{\epsilon(\mu\oplus\nu)}})^T\chi_2^{\epsilon(\mu\oplus\nu)}\big)
  =\tr\frac{d}{d\epsilon}\vert_{\epsilon=0}(\overline{\chi_2^{\epsilon(\mu\oplus\nu)}})^T\chi_2^{\epsilon(\mu\oplus\nu)}
\end{equation*}
Thus the above multiple of the identity is zero. This show the first statement in the lemma. To obtain the second equation, we compute that
\begin{equation*}
  0=\bar\partial\frac{d}{d\epsilon}\vert_{\epsilon=0}(\overline{\chi_2^{\epsilon(\mu\oplus\nu)}})^T\chi_2^{\epsilon(\mu\oplus\nu)}=\frac{d}{d\epsilon}\vert_{\epsilon=0}(\overline{\partial\chi_2^{\epsilon(\mu\oplus\nu)}})^T+\frac{d}{d\epsilon}\vert_{\epsilon=0}\bar\partial\chi_2^{\epsilon(\mu\oplus\nu)}
\end{equation*}
and then the second claim follows from
\begin{equation*}\frac{d}{d\epsilon}\vert_{\epsilon=0}\bar\partial\chi_2^{\epsilon(\mu\oplus\nu)}=\frac{d}{d\epsilon}\vert_{\epsilon=0}(\epsilon\mu_1-\frac{1}{2}\tilde g^{-1}_X\tr(\epsilon\nu_1)^2)\partial\chi_2^{\epsilon(\mu\oplus\nu)}+\epsilon\nu_1\chi_2^{\epsilon(\mu\oplus\nu)}=\nu_1.
\end{equation*}

\end{proof}

On the space of all families of operators
\begin{align*}
  F^{\epsilon(\mu\oplus\nu)}:L^2(X^{\epsilon(\mu\oplus\nu)},&\ \text{End}E^{\epsilon(\mu\oplus\nu)}\otimes\Omega_j(X^{\epsilon(\mu\oplus\nu)}))\\ &\to L^2(X^{\epsilon(\mu\oplus\nu)},\text{End}E^{\epsilon(\mu\oplus\nu)}\otimes\Omega_i(X^{\epsilon(\mu\oplus\nu)}))
\end{align*} 
we define the connection $L_{\mu_1\oplus\nu_1}$ as follows
 $$ L_{\mu_1\oplus\nu_1} F=\frac{d}{d\epsilon}\vert_{\epsilon=0}(\text{Ad}\chi_2^{\epsilon(\mu\oplus\nu)})^{-1} (\chi_1^{\epsilon(\mu_1\oplus\nu_1)})_* F^{\epsilon(\mu_1\oplus\nu_1)} (\chi_1^{\epsilon(\mu_1\oplus\nu_1)})_*^{-1}\text{Ad}\chi_2^{\epsilon(\mu\oplus\nu)}$$  
 where $(\chi_1^{\epsilon(\mu_1\oplus\nu_1)})_*$ refers to pull back of forms. We also write $\bar L_{\mu_1\oplus\nu_1}$ when we differentiate along anti-holomorphic vector fields, e.g. when we replace $\epsilon$ by $\bar\epsilon$ derivatives above.
\begin{lemma}\label{lem:opdif} 
  We have the following formula for $\bar\partial$ operator acting on sections of $\textnormal{End}E$ and for $\bar\partial^*$ acting on $\textnormal{End}E$ valued $(0,1)$-forms
  \begin{align*}
   L _{\mu_1\oplus\nu_1}\bar\partial&=\textnormal{ad}{\nu_1}-\mu\partial, \qquad L_{\mu_1\oplus\nu_1}\bar\partial=0,\\
   L_{\mu_1\oplus\nu_1}\bar\partial^*&=0,\qquad L_{\mu_1\oplus\nu_1}\bar\partial^*=-\star\textnormal{ad}\nu_1\star -\partial^* \bar\mu_1.
  \end{align*}
\end{lemma}
\begin{proof}
We show the first identity as follows
  \begin{align*}
    \frac{d}{d\epsilon}\vert_{\epsilon=0}(\text{Ad}\chi_2^{\epsilon(\mu\oplus\nu)})^{-1} (\chi_1^{\epsilon(\mu_1\oplus\nu_1)})_* \bar\partial^{\epsilon(\mu_1\oplus\nu_1)} (\chi_1^{\epsilon(\mu_1\oplus\nu_1)})_*^{-1}\text{Ad}\chi_2^{\epsilon(\mu\oplus\nu)}=\\ \frac{d}{d\epsilon}\vert_{\epsilon=0}(\text{Ad}\chi_2^{\epsilon(\mu\oplus\nu)})^{-1} (\chi_1^{0(\mu_1\oplus\nu_1)})_* \bar\partial (\chi_1^{0(\mu_1\oplus\nu_1)})_*^{-1}\text{Ad}\chi_2^{\epsilon(\mu\oplus\nu)}\\+\frac{d}{d\epsilon}\vert_{\epsilon=0}(\text{Ad}\chi_2^{0(\mu\oplus\nu)})^{-1} (\chi_1^{\epsilon(\mu_1\oplus\nu_1)})_* \bar\partial (\chi_1^{\epsilon(\mu_1\oplus\nu_1)})_*^{-1}\text{Ad}\chi_2^{0(\mu\oplus\nu)}.
  \end{align*}
From \cite{ZTpuncRie}[Equation (2.6)] we know that the variation $\frac{d}{d\epsilon}\vert_{\epsilon=0}(\chi_1^{\epsilon(\mu_1\oplus\nu_1)})_* \bar\partial (\chi_1^{\epsilon(\mu_1\oplus\nu_1)})_*^{-1}=-\mu\partial$, since the computation is the same just using the Beltrami differential $\mu_1-\frac{1}{2}\tilde g_{X_0}\tr\nu_1^2$ and $$\frac{d}{d\epsilon}\vert_{\epsilon=0}\epsilon\mu_1-\frac{1}{2}\tilde g_{X_0}\tr(\epsilon\nu_1)^2=\mu_1.$$
For the second term we find that using in the second equality the defining differential equation for $\chi_2$
\begin{align*}
  \frac{d}{d\epsilon}\vert_{\epsilon=0}&(\text{Ad}\chi_2^{\epsilon(\mu\oplus\nu)})^{-1} \bar\partial \text{Ad}\chi_2^{\epsilon(\mu\oplus\nu)}=\frac{d}{d\epsilon}\vert_{\epsilon=0}\bar\partial-(\text{ad}\chi_2^{\epsilon(\mu\oplus\nu)})^{-1}\bar\partial\chi_2^{\epsilon(\mu\oplus\nu)}\\ &=\frac{d}{d\epsilon}\vert_{\epsilon=0}\text{ad}\left((\epsilon\mu_1-\tilde g_X^{-1}\tr(\epsilon\nu_1)^2)(\chi_2^{\epsilon(\mu\oplus\nu)})^{-1}\partial\chi_2^{\epsilon(\mu\oplus\nu)}+\epsilon\nu_1\right)=\text{ad}\nu_1.
\end{align*}
Putting these two equations together shows the first equality, the rest is shown similarly.
\end{proof}

\section{Derivatives of the metric}
We now want to study the metric in our local coordinates. Pick a base point $([X],[E])\in \M$ and choose a basis of $H^1(X,TX) \oplus H^1(X, {\rm End}E)$ which is orthonormal. 
The metric is given by the following expression at some $\mu\oplus\nu\in H^1(X,TX) \oplus H^1(X, {\rm End}E)$ small enough by using the Kodaira Spencer map
\begin{align*}
  g_{\mu\oplus\nu}&(\mu_1\oplus\nu_1,\mu_2\oplus\nu_2)\\=i&\int_\Sigma P_{TX}\left((\chi_1^{\mu\oplus\nu})^{-1}_*\frac{\mu_1-\tilde g^{-1}_{X_{\mu\oplus\nu}}\tr\nu_1\nu}{1-\vert\mu\vert^2}\right)\overline{\left((\chi_1^{\mu\oplus\nu})^{-1}_*\frac{\mu_2-\tilde g^{-1}_{X_{\mu\oplus\nu}}\tr\nu_2\nu}{1-\vert\mu\vert^2}\right)}\tilde g_{X_{\mu\oplus\nu}} \\ &+i \int_\Sigma\tr\left(P_{\textnormal{End}E}\left(\text{Ad}\left((\chi_1^{\mu\oplus\nu})^{-1}_*\chi_2^{\mu\oplus\nu}\right)\left(\nu_1+(\mu_1-\tilde g^{-1}_{X_{\mu\oplus\nu}}\tr\nu_1\nu)\chi_2^{\mu\oplus\nu}\partial_\SetH\chi_2^{\mu\oplus\nu}\right)\right)\right.\\ &\left.\wedge\overline{P_{\textnormal{End}E}\left(\text{Ad}\left((\chi_1^{\mu\oplus\nu})^{-1}_*\chi_2^{\mu\oplus\nu}\right)\left(\nu_2+(\mu_2-\tilde g^{-1}_{X_{\mu\oplus\nu}}\tr\nu_2\nu)\chi_2^{\mu\oplus\nu}\partial_\SetH\chi_2^{\mu\oplus\nu}\right)\right)}\right).
\end{align*}
By using that the harmonic projections are projections we find that
\begin{align*}
  g_{\mu\oplus\nu}&(\mu_1\oplus\nu_1,\mu_2\oplus\nu_2)\\=i&\int_\Sigma P_{TX}\left((\chi_1^{\mu\oplus\nu})^{-1}_*\frac{\mu_1-\tilde g^{-1}_{X_{\mu\oplus\nu}}\tr\nu_1\nu}{1-\vert\mu\vert^2}\right)\overline{(\chi_1^{\mu\oplus\nu})^{-1}_*\frac{\mu_2-\tilde g^{-1}_{X_{\mu\oplus\nu}}\tr\nu_2\nu}{1-\vert\mu\vert^2}}\tilde g_{X_{\mu\oplus\nu}} \\ &+i \int_\Sigma\tr\left(\left(\text{Ad}\overline{(\chi_2^{\mu\oplus\nu})}^T\text{Ad}\chi_2^{\mu\oplus\nu}\right)((\text{Ad}\chi_2^{\mu\oplus\nu})^{-1}(\chi_1^{\mu\oplus\nu})_* \right.\\ &\quad P_{\textnormal{End}E}\left((\chi_1^{\mu\oplus\nu})^{-1}_*\text{Ad}\chi_2^{\mu\oplus\nu}\left(\nu_1+(\mu_1-\tilde g^{-1}_{X_{\mu\oplus\nu}}\tr\nu_1\nu)\chi_2^{\mu\oplus\nu}\partial_\SetH\chi_2^{\mu\oplus\nu}\right)\right)\\ &\qquad\wedge\overline{\nu_2+(\mu_2-\tilde g^{-1}_{X_{\mu\oplus\nu}}\tr\nu_2\nu)\chi_2^{\mu\oplus\nu}\partial_\SetH\chi_2^{\mu\oplus\nu}}.
  \end{align*}
\begin{lemma}
  The first derivatives of the metric in the local coordinates vanishes, e.g. the coordinates are K{\"a}hler coordinates.
\end{lemma}
\begin{proof}
We will now show that $$  \frac{d}{d\epsilon}\vert_{\epsilon=0} g_{\epsilon(\mu\oplus\nu)}(\mu_1\oplus\nu_1,\mu_2\oplus\nu_2)=0.$$
We observe that the $L_{\mu\oplus\nu}$-derivaties of the  projections $P_{\textnormal{End}E}$ are given by $P(L_{\mu\oplus\nu}\bar\partial)\Delta_0^{-1}\bar\partial^*$ and therefore they annihilate harmonic forms. Also we have that 
$$\frac{d}{d\epsilon}\mid_{\epsilon=0}\left(\text{Ad}\overline{(\chi_2^{\epsilon(\mu\oplus\nu}))}^T\text{Ad}\chi_2^{\epsilon(\mu\oplus\nu})\right)=0$$ 
from Lemma~\ref{lemchiprop}. Since $\tr\nu_1\nu(\chi_2^{\mu\oplus\nu})^{-1}\partial \chi_2^{\mu\oplus\nu}$ vanish to second order at $0$ it does not contribute. And so we have that the only contribution is
\begin{align*}
  \frac{d}{d\epsilon}g_{\epsilon(\mu\oplus\nu)}(\mu_1\oplus\nu_1,\mu_2\oplus\nu_2)=i&\int_\Sigma (-\tilde g^{-1}_{X_{\mu\oplus\nu}}\tr\nu_1\nu)\bar\mu_2 \tilde g_{X_{\mu\oplus\nu}} \\ &+i \int_\Sigma\tr(\nu_1\wedge\overline{\mu_2\bar\nu^T}^T)=0.
\end{align*}
The $\bar\epsilon$-derivative is calculated similarly.
\end{proof}
\begin{thm}
 We have the following formula for the second order derivatives of the metric
  \begin{align}
\left.    \frac{d^2}{d\epsilon_1d\bar\epsilon_2}\right\vert_{\epsilon=0}&g_{\epsilon{(\mu\oplus\nu)}}(\mu_3\oplus\nu_3,\mu_4\oplus\nu_4)\nonumber\\ &=-i \int_\Sigma \tr((-\mu_1\partial+\text{ad}\nu_1)\Delta_0^{-1}(-\partial^*\bar\mu_2-\star\text{ad}\nu_2\star)\nu_3 \wedge\bar\nu_4^T\nonumber\\ &-i \int_\Sigma \tr(\text{ad}\Delta^{-1}_0((-\star)\text{ad}\nu_2\star\nu_1-\star(\partial\mu_1\bar\nu_2^T)-\star(\bar\partial\bar\mu_2\nu_1))\nu_3\wedge\bar\nu_4^T)\nonumber\\ &-i\int_\Sigma \mu_1\bar\mu_2\tr\nu_3\wedge\bar\nu_4^T\nonumber\\ &-i\int_\Sigma \tr(\text{ad}\nu_1+\mu_1\partial)\Delta_o^{-1}\bar\partial^* \mu_3\bar\nu_2^T\wedge\bar\nu_4^T
\nonumber\\ &-i\int_\Sigma \tr\mu_3(\partial\Delta_0^{-1}(\star[\star v_1\nu_2]-\star(\partial\mu_1\bar\nu_2^T)-\star(\bar\partial\bar\mu_2\nu_1))\wedge\bar\nu_4^T\nonumber\\ &-i\int_\Sigma \tr \bar\mu_2\mu_3\nu_1\wedge\overline{\nu_4}^T \nonumber\\ &-i\int_\Sigma \tr\bar\partial\Delta_o^{-1}(-\star\text{ad}\nu_2\star-\partial^*\mu_2) \nu_3\wedge\bar\mu_4\nu_1
\nonumber\\ &-i\int_\Sigma \tr\nu_3\wedge \overline{\mu_4(\partial\Delta_0^{-1}(\star[\star v_2\nu_1]-\star(\partial\mu_2\bar\nu_1^T)-\star(\bar\partial\bar\mu_1\nu_2))}^T\nonumber\\ &-i\int_\Sigma \tr\nu_3\wedge\overline{\bar\mu_1\mu_4\nu_2}^T-i\int_\Sigma \tr\mu_3\nu_1\wedge\overline{\mu_4\nu_2}^T
\nonumber\\
&-i\int_\Sigma(\mu_1\bar\mu_2+(\Delta_0+\frac{1}{2})^{-1}(\mu_1\bar\mu_2\tilde g_{X_{\mu\oplus\nu}}^{-1}))\tilde g_{X_{\mu\oplus\nu}} \nonumber\\ &-i\int_\Sigma(\mu_1\partial^*(\Delta_{0\numbC})^{-1}\bar\partial \bar\mu_2 \mu_3)\bar\mu_4\tilde g_{X_{\mu\oplus\nu}}\label{metric}.
  \end{align}
\end{thm}

This formula follows from \cite{AP}, since these new coordinates are related to our coordinates in \cite{AP} modulo a quadratic holomorphic coordinate change, up to second order. Alternatively the same computations can be done using the properties of \ref{sec:thetaprop}.

We can now conclude that the Ricci form is given by
\begin{thm}
  \begin{align*}
    \text{Ric}^{1,1}(\mu_1\oplus\nu_1,\bar\mu_2\oplus\bar\nu_2^T)&=-i\tr(-(\mu_1\bar\mu_2+\mu_1\partial\Delta_0^{-1}\partial^*\bar\mu_2)P_{TX})\\ &-i\tr_E((\text{ad}\Delta_{0,E}^{-1}(\star[\star\nu_1,\nu_2])-\text{ad}\nu_1\Delta_{0,E}^{-1}\star\text{ad}\nu_2\star)P_{\textnormal{End}E})\\ &-i\tr_E((\text{ad}\Delta_{0,E}^{-1}(\partial^*\bar\mu_1\bar\nu_2^T)+\mu_1\partial\Delta_{0,E}^{-1}\star\text{ad}\nu_2\star)P_{\textnormal{End}E})\\ &-i\tr_E((\text{ad}\Delta_0^{-1}(\bar\partial^*\bar\mu_2\nu_1)+\text{ad}\nu_1\Delta_{0,E}^{-1}\partial^*\bar\mu_2)P_{\textnormal{End}E})\\ &-i\tr_E(\mu_1 \bar P_{\textnormal{End}E}\bar\mu_2P_{\textnormal{End}E}),
  \end{align*}  
where $\bar P_{\textnormal{End}E}$ is the projection on harmonic $(1,0)$-forms with values in ${\rm End}E$.
\end{thm}
\begin{proof}
We denote the martix that represent the metric in our local coordinates by $G$. Since we have K{\"a}hler coordinates we know that $G\vert_{0\oplus0}=I$ and that $\frac{d}{d\epsilon}\vert_{\epsilon=0}G_{\epsilon(\mu\oplus\nu)}=0$ and so
$$\text{Ric}^{1,1}=-i\partial\bar\partial \log\det G$$
and so we need to calculate:
$$\frac{d^2}{d\epsilon d\bar\epsilon}\vert_{\epsilon=0}\log\det G=\tr G^{-1}\frac{d^2}{d\epsilon d\bar\epsilon}\vert_{\epsilon=0} G.$$  
Here only the diagonal terms contribute ($\nu_3=\nu_4$ and $\mu_1=\mu_2$), which are the first three terms in \eqref{metric} and the last two terms. To illustrate how this work we look at the first integral.

We have that $P_{{\rm End}E} \nu=\sum_{i} g(\nu,\nu_i)\nu_i$, since we chose $\nu_i$ to be a orthonormal basis. And so we have that $\sum_{i} g(F(\nu),\nu_i)=\tr(F P_{{\rm End} E})$. This imply that
\begin{align*}
 -i \int_\Sigma \tr((-\mu_1\partial+\text{ad}\nu_1)\Delta_0^{-1}(\partial^*\bar\mu_2-\star\text{ad}\nu_2\star)\nu_3 \wedge\bar\nu_4^T \\=\tr ((-\mu_1\partial+\text{ad}\nu_1)\Delta_0^{-1}(-\partial^*\bar\mu_2-\star\text{ad}\nu_2\star)P_{{\rm End} E}).
\end{align*}
Multiplying out the parethesis we get the last term in the middel three traces in the theorem and additionally the term $\tr ((-\mu_1\partial)\Delta_0^{-1}(-\partial^*\bar\mu_2)P_{{\rm End} E})$, which is a component in the last trace, since 
$$\mu_1P_{\textnormal{End}E}\bar\mu_2=\mu_1\bar\mu_2 Id-((-\mu_1\partial)\Delta_0^{-1}(-\partial^*\bar\mu_2).$$
Repeating with all the remaining integrals and collection together terms we get the idendity for the Ricci form.
\end{proof}
 \section{Ricci potential}
\label{sec:pot}

Recall that the manifold $\M'\cong \T\times M'$ is equipped with the K\"{a}hler structure, where the symplectic structure is a sum of the pull back of two symplectic forms, namely the Weil-Petersen symplectic form on $\mathcal{T}$ and  the Seshadri-Atiyah-Bott-Goldmann symplectic form on $M'$. For each $\sigma\in \T$,  we know that $M'$ equipped with the Seshadri-Atiyah-Bott-Goldmann symplectic form and the complex structure induced by $\sigma$, the Ricci potential is $\log\det \Delta_{\text{Ad}E}$. For Teichm\"{u}ller space, the Ricci potential is $\log\det \Delta_0$ where $\Delta_0$ is the Laplace-operator on function on $\Sigma$ of course depending on $\sigma\in \T$.

In the following, when we vary the determinant of the Laplace operator, we will express it as integrals. For this purpose we use the integral kernel of the Laplace operator on functions with values in Ad$E$. We consider the kernel as an equivariant function on the cover $G:\SetH\times\SetH\to \text{Ad}\numbC^n$. To make the integrals converge, the singularity of $G$ on the diagonal will be cancels by the kernel $Q:\SetH\times\SetH\to\numbC^{n^2-1}$. This is the kernel of the Laplace operator on $\SetH$ with values in $\numbC^{n^2-1}$ and is given by $$ Q(z,z')=-\frac{1}{2\pi}\log\left(\frac{\vert z-z'\vert}{\vert z-\bar z'\vert}\right){\rm Id}_{\numbC^{n^2-1}}.$$ We will also write our integrals over $\Sigma$. This makes sense, since the kernels are defined on $\Sigma$, but an other interpretation is to integrate over a fundamental domain in $\SetH$.

\begin{lemma}
\label{lem:firstderiv}
  The first order derivativs of $\log\det \Delta_{\text{Ad}E}$ in the direction $\nu\in\mathcal{H}^{(0,1)}(X,\text{Ad}E)$ (and its complex conjugate) at $(X,E)\in\mathcal{M}$ are
  \begin{align*}
    \partial_\nu\log\det \Delta_{\text{Ad}E}=-i\int_\Sigma\tr{\rm ad}\nu\wedge\left(\partial'(G(z,z')-Q(z,z'))\right)\vert_{z=z'},\\
\bar\partial_{\bar\nu^T}\log\det \Delta_{\text{Ad}E}=i\int_\Sigma\tr{\rm ad}\bar\nu^T\wedge \left(\partial'(G(z,z')-Q(z,z'))\right)\vert_{z=z'}.
\end{align*}
Moreover, for $\mu\in H^1(X,T)$, we have that
$$
\partial_\mu\log\det \Delta_{\text{Ad}E}=-i\int_\Sigma\mu\tr(\partial\partial'(G(z,z')-Q(z,z')))\vert_{z=z'}.
$$
\end{lemma}
\begin{proof}
The first equation follows from \cite[Lemma 3]{ZTVB}, since the coordinates agree, when we stay in the fiber over $X$. The second equation follows since $\log\det \Delta_{\text{Ad}E}$ is real by the self adjointness of $\Delta_{\text{Ad}E}$. 

The last equation we can calculated similar to the verification of \cite[Lemma 3]{ZTpuncRie}, keeping in mind that we need to work with the Selberg zeta function $Z(\rho_\SetH,\rho_{\text{Ad}E},s)$ and using our coordinates agree with Bers' coordinates on $\mathcal{T}\times{\rho_E}$.
\end{proof}

Now the second order derivatives can be calculated as follows.
\begin{thm}
  Second order variation of $\log\det \Delta_{{\rm Ad}E}$ are
  \begin{align*}
    \bar\partial_{\bar\nu_2^T}\partial_{\nu_1} \log\det \Delta_{{\rm Ad} E}&=\tr(({\rm ad}(\Delta_{0,E}^{-1}\star[\star\nu_1,\nu_2]-{\rm ad}\nu_1\Delta_{0,E}^{-1}\star {\rm ad}\nu_2\star)P_{\textnormal{End}E})\\\quad&-\frac{2ni}{2\pi}\omega_{M'}(\nu_1,\nu_2),\\
   \bar\partial_{\bar\nu_2^T}\partial_{\mu_1} \log\det \Delta_{{\rm Ad} E}&=\tr(({\rm ad}(\Delta_{0,E}^{-1}\bar\partial^*\mu_1\bar\nu_2^T)+\mu\partial\Delta_{0,E}^{-1}\star{\rm ad}\nu_2\star)P_{\textnormal{End}E}), \\
   \bar\partial_{\bar\mu_2}\partial_{\mu_1} \log\det \Delta_{{\rm Ad} E}&=-\tr(\mu_1 \bar P_{{\rm End}E}\bar\mu_2P_{{\rm End}E})-\frac{(n^2-1)i}{6\pi}\omega_{\mathcal{T}}(\mu_1,\mu_2).
  \end{align*}
\end{thm}
\begin{proof}
  The formula for $\bar\partial_{\bar\nu_2^T}\partial_{\nu_1} \log\det \Delta_{{\rm Ad} E}$ follows from \cite[Theorem 2]{ZTVB}, since the coordiantes agree with their coordinats up to a second order holomorphic coordinate change.

For the second equation recall that $\mu^{0\oplus\epsilon\nu}$ is represented by $$\mu\oplus P_{{\rm End}E}^{0\oplus\epsilon\nu}{\rm Ad}\Psi^{0\oplus\epsilon\nu}_2(\mu(\Psi_2^{0\oplus\epsilon\nu})^{-1}\partial\Psi_2^{0\oplus\epsilon\nu})$$ and so we have from Lemma \ref{lem:firstderiv} that
\begin{align*}
 i \bar\partial_{\bar\nu_2^T}\partial_{\mu_1}& \log\det \Delta_{{\rm Ad} E}\\ &=\frac{d}{d\bar\epsilon}\vert_{\epsilon=0}\Big(\int_\Sigma \mu_1 \tr(\partial\partial'(G(z,z')-Q(z,z'))\vert_{z=z'}\\ &\quad+\int_\Sigma  {\rm ad}P_{\textnormal{End}E}{\rm Ad}\chi^{0\oplus\epsilon\nu}_2(\mu(\chi^{0\oplus\epsilon\nu})_2^{-1}\partial\chi^{0\oplus\epsilon\nu}_2)\\ &\qquad\wedge(\partial'(G(z,z')-Q(z,z')))\vert_{z=z'}\Big).
\end{align*}
Now conjucation by $\chi_2^{0\oplus\epsilon\nu}(z)$ under the trace and moving under the evaluation on the diagonal with different variables we find that
\begin{align*}
 i \bar\partial_{\bar\nu_2^T}\partial_{\mu_1}& \log\det \Delta_{{\rm Ad} E}\\ &=\frac{d}{d\bar\epsilon}\vert_{\epsilon=0}\Big(\int_\Sigma \mu_1 \tr(\chi_2^{0\oplus\epsilon\nu}(z)\chi_2^{0\oplus\epsilon\nu}(z')^{-1}\partial\partial'(G(z,z')-Q(z,z'))\vert_{z=z'}\\ &\quad+\int_\Sigma  {\rm ad}P_{\textnormal{End}E}{\rm Ad}\chi^{0\oplus\epsilon\nu}_2(\mu(\chi^{0\oplus\epsilon\nu})_2^{-1}\partial\chi^{0\oplus\epsilon\nu}_2)\\ &\qquad\wedge(\partial'(G(z,z')-Q(z,z')))\vert_{z=z'}\Big).
\end{align*}
Now the second term vanish unless we differentiate $\partial\chi^{0\oplus\epsilon\nu}_2$ and for the first term we have that we need to consider 
$$-i\chi_2^{0\oplus\epsilon\nu}(z)\partial(\chi_2^{0\oplus\epsilon\nu}(z))^{-1}\chi_2^{0\oplus\epsilon\nu}(z)\chi_2^{0\oplus\epsilon\nu}(z')^{-1}\partial'(G(z,z')-Q(z,z')).$$ 
By Lemma \ref{lem:opdif} we find that the $\bar\epsilon$ derivative of $\chi_2^{0\oplus\epsilon\nu}(z)\partial(\chi_2^{0\oplus\epsilon\nu}(z))^{-1}$ is ${\rm ad}\bar\nu_2^T\partial'(G(z,z')-Q(z,z'))$. For $\chi_2^{0\oplus\epsilon\nu}(z)\chi_2^{0\oplus\epsilon\nu}(z')^{-1}\partial'(G(z,z')-Q(z,z'))$ an explicit calculation like in the proof of \cite[Theorem 1]{ZTVB} give that  $$\chi_2^{0\oplus\epsilon\nu}(z)\chi_2^{0\oplus\epsilon\nu}(z')^{-1}\partial'Q(z,z')=\frac{{\rm ad}\bar\nu_2^T}{2\pi}$$ which has trace $0$ and so the variation of the term $-i\chi_2^{0\oplus\epsilon\nu}(z)\chi_2^{0\oplus\epsilon\nu}(z')^{-1}\partial'G(z,z')$ is finite and by the relations in Lemma \ref{lem:opdif} we find the $\bar\epsilon $-derivative is $-\int_\Sigma G(z,z'')\star{\rm ad}\nu_2(z'')\star P(z'',z)^{0,1}$, which is the kernel of $-\Delta_0^{-1}\star{\rm ad}\nu_2\star P_{{\rm End}E}$. This leads us to
\begin{align*}
  i\bar\partial_{\bar\nu_2^T}\partial_{\mu_1} \log\det \Delta_{{\rm Ad} E} &=\left(\int_\Sigma\right. \mu \tr({\rm ad}\bar\nu_2^T\wedge\partial'(G(z,z')-Q(z,z')\\ &\quad+\partial\int_\Sigma G(z,z'')\star{\rm ad}\nu_2(z'')\star P(z'',z)^{0,1})\vert_{z=z'}\\ &\quad-\left.\int_\Sigma {\rm ad} P_{\textnormal{End}E}(\mu\bar\nu_2^T)\wedge(\partial'(G(z,z')-Q(z,z')))\vert_{z=z'}\right).
\end{align*}
Now using that $P_{\textnormal{End}E}(\mu\bar\nu_2^T)= (I-\bar\partial\Delta_{0,E}^{-1}\bar\partial^*\mu\bar\nu_2^T)$ and that $\partial\int_\Sigma G(z,z'')\star{\rm ad}\nu_2(z'')\star P(z'',z)^{0,1})\vert_{z=z'}$ is finite on the diagonal, we find that
\begin{align*}
  \bar\partial_{\bar\nu_2^T}\partial_{\mu_1} \log\det \Delta_{{\rm Ad} E} &=+i\int_\Sigma {\rm ad}(\bar\partial\Delta_{0,E}^{-1}\bar\partial^*\mu\bar\nu_2^T)\wedge(\partial'(G(z,z')-Q(z,z')))\vert_{z=z'}\\ &\quad+\tr(\mu\partial\Delta_{0,E}^{-1}\star{\rm ad}\nu_2\star P_{\textnormal{End}E}).
\end{align*}
Now we can move $\bar\partial$-operator in the first term past the wedge product to get a term of the form
\begin{align*}
  \bar\partial(\partial(G-Q))\vert_{z=z'}=((\bar\partial+\partial)\partial'(G-Q))\vert_{z=z'}.
\end{align*}
In \cite{ZTVB} the Q terms are calculated and shown to be a multiple of $I_{{\rm Ad}E}$ and further we have that $-i\bar\partial'\partial' G=0$, since there are no holomorphic sections of ${\rm Ad}E$. Finally $-\bar\partial\partial' G=P(z,z')$ when $z\neq z'$. And so we get that
\begin{align*}
  \bar\partial_{\bar\nu_2^T}\partial_{\mu_1} \log\det \Delta_{{\rm Ad} E}=\tr((\Delta_{0,E}^{-1}\bar\partial^*\mu\bar\nu_2^T)P_{\textnormal{End}E}+\mu\partial\Delta_{0,E}^{-1}\star{\rm ad}\nu_2\star P_{\textnormal{End}E}).
\end{align*}

Finally $\bar\partial_{\bar\mu_2}\partial_{\mu_1} \log\det \Delta_{{\rm Ad} E}$ can be calculated as follows
\begin{align*}
  \bar\partial_{\bar\mu_2}\partial_{\mu_1} \log\det \Delta_{{\rm Ad} E} &=\frac{d}{d\bar\epsilon}\vert_{\epsilon=0}\int_\Sigma P_{TX}\left((\chi_1^{\epsilon\mu_2\oplus0})^{-1}_*\frac{\mu_1}{1-\vert\mu\vert^2}\right) \\ &\quad\tr(\partial^{\epsilon\mu_2\oplus0}(\partial')^{\epsilon\mu_2\oplus0}(G^{\epsilon\mu_2\oplus0}(z,z')-Q^{\epsilon\mu_2\oplus0}(z,z'))\vert_{z=z'}\\ &+\frac{d}{d\bar\epsilon}\vert_{\epsilon=0}\int_\Sigma  P_{\textnormal{End}E}\left((\chi_1^{\epsilon\mu_2\oplus0})^{-1}_*\text{Ad}\chi_2^{\epsilon\mu_2\oplus0}\mu_1\chi_2^{\epsilon\mu_2\oplus0}\partial_\SetH\chi_2^{\epsilon\mu_2\oplus0}\right)\\ &\qquad\wedge(\partial'(G(z,z')-Q(z,z')))\vert_{z=z'}.
\end{align*}
From \ref{lemchiprop} we have that the second term vanish since both $\partial\chi_2^{\epsilon\mu_2\oplus0}$ and $\frac{d}{d\bar\epsilon}\vert_{\epsilon=0}\partial\chi_2^{\epsilon\mu_2\oplus0}$ vanish. For the first term we make a change of coordinates with $(\chi_1^{\mu_2\oplus0})^{-1}$, then $\partial^{\epsilon\mu_2\oplus0}(\partial')^{\epsilon\mu_2\oplus0}(G^{\epsilon\mu_2\oplus0}(z,z')-Q^{\epsilon\mu_2\oplus0}(z,z'))$ becomes $$(\chi_1^{\mu_2\oplus0})^{-1}_*\partial\partial'(G(\chi_1^{\mu_2\oplus0}(z),\chi_1^{\mu_2\oplus0}(z'))-Q(\chi_1^{\mu_2\oplus0}(z),\chi_1^{\mu_2\oplus0}(z'))).$$ 
After that we conjugate by $\chi_2^{\epsilon\mu_2\oplus0}(z)$ under the trace and move them under the evalutation on the diagonal with different variable, this gives when we do the $\bar\epsilon$ differentiation $\bar L_{\mu_2\oplus0}$ of the kernel the following formula
\begin{align*}
i\bar\partial_{\bar\mu_2}\partial_{\mu_1} \log\det \Delta_{{\rm Ad} E}&=\int_\Sigma \mu_1\tr(\bar L_{\mu_2\oplus0}(\partial\partial'(G(z,z')-Q(z,z')))\vert_{z=z'}\\ &+\int_\Sigma (\bar L_{\mu_2\oplus0}P_{TX}\left((\chi_1^{\epsilon\mu_2\oplus0})^{-1}_*\frac{\mu_1}{1-\vert\mu\vert^2}\right)) \tr(\partial\partial'(G(z,z')-Q(z,z'))\vert_{z=z'}
\end{align*}

We then find that $\tr L_{\mu_2\oplus 0} \partial\partial' Q(z,z')=\frac{\bar\mu_2(n^2-1)}{12 \pi y^2}$, since our $Q$ is just $I_{{\rm Ad}E}$ times the $Q$ from \cite[section 4.4]{ZTRie}, where the computation is done. This means that $L_{\mu_2\oplus 0} -i\partial\partial' G(z,z')$ is finite as well and since it is the kernel of the operator $\partial\Delta_{0,E}^{-1}\bar\partial^*$, we see that the variation is $$L_{\mu_2\oplus 0} -i\partial\partial'G(z,z')=\bar\mu_2(I-P_{\textnormal{End}E})+\partial\Delta_{0,E}^{-1}\partial^*\bar\mu_2 P_{\textnormal{End}E}=\mu_2 I-(\bar P_{{\rm End}E}\bar\mu_2P_{{\rm End}E}).$$ Since we know $ \int_\Sigma \mu_1 \tr(\bar L_{\mu_2\oplus0}(\partial\partial'(G(z,z'))\vert_{z=z'}$ is finite, we conclude that $$-i\int_\Sigma\mu_1 \tr(\bar L_{\mu_2\oplus0}(\partial\partial'(G(z,z'))\vert_{z=z'})=\tr(\mu_1\bar P_{{\rm End}E}\bar\mu_2 P_{\textnormal{End}E}).$$ 
Finally, we have the term with 
\begin{align*}
\bar L_{\mu_2\oplus0}\mu_1&=\frac{d}{d\bar\epsilon}\vert_{\epsilon=0}P_{TX}\left(\frac{\mu_1}{1-\vert\mu-\frac{1}{2}\tilde g_X^{-1}\tr\nu^2\vert^2}\frac{\partial\chi_1^{\mu\oplus\nu}}{\overline{\partial\chi_1^{\mu\oplus\nu}}}\right)\circ(\chi_1^{\mu\oplus\nu})^{-1}\\ &=\bar\partial\Delta^{-1}\partial^*(\bar\mu_2\mu_1).
\end{align*} 
We can now use Stokes theorem to move the $\bar\partial$ from this term to $\partial\partial' (G-Q)\vert_{z=z'}$ and we get that
$$(\bar\partial+\bar\partial')\partial\partial' (G-Q)\vert_{z=z'}=(\partial P_{\textnormal{End}E}(z,z'))\vert_{z=z'}=0,$$ 
since the harmonic forms are in the kernel of $\partial$. Now gathering the terms we get the last formula.
\end{proof}
For the function $\log\det \Delta_0$, we see that there is a holomorphic coordinate change of second order from the Bers coordinates used in \cite{ZTpuncRie} to the relevant part of our holomorphic coordinates (modulo second order) on the universal moduli space and so $\bar\partial\partial \log\det \Delta_0$ can be given in our coordinates as follows
\begin{align*}
  \bar\partial_{\mu_2\oplus\nu_2}\partial_{\mu_1\oplus\nu_1} \log\det \Delta_0=\frac{i}{6\pi}\omega_{\mathcal{T}}(\mu_1,\mu_2)-\tr((\mu_1\bar\mu_2-\mu_1\partial\Delta_0^{-1}\partial^*\bar\mu_2)P_{TX}).
\end{align*}
From this formula we get Theorem \ref{thmRP}.

\end{document}